\documentclass[12pt,reqno]{amsart}
\usepackage{palatino}
\usepackage{amssymb}
\usepackage{graphicx}
\usepackage{geometry}
\newcommand{\bex}{\begin{eqnarray*}}
\newcommand{\eex}{\end{eqnarray*}}
\newcommand{\be}{\begin{eqnarray}}
\newcommand{\ee}{\end{eqnarray}}

\newtheorem{remark}{Remark}

\newtheorem{theorem}{Theorem}[section]

\newtheorem{corollary}[theorem]{Corollary}
\newtheorem{lemma}[theorem]{Lemma}
\newtheorem{definition}[theorem]{Definition}

\begin{document}
\title[ Polynomial approximation in higher-order weighted Dirichlet spaces]
	{Polynomial approximation in higher-order weighted Dirichlet spaces}

	\author{Li He}
	\address{Li He: School of Mathematics and Information Science, Guangzhou University, Guangzhou 510006, China.}
	\email{helichangsha1986@163.com}
	
	\author{Yuanhao Yan*}
	\address{Yuanhao Yan: School of Mathematics and Information Science, Guangzhou University, Guangzhou 510006, China.}
	\email{18322912287@163.com}



	\thanks{Li He is supported by NNSF of China (Grant No. 12371127)}
    \thanks{*Corresponding author, e-mail: 18322912287@163.com}

\begin{abstract}

Fej\'er's theorem guarantees norm convergence of Ces\`aro means of Taylor partial sums in the Hardy space, whereas such convergence generally fails in weighted Dirichlet-type spaces, especially in the higher-order setting.
In this paper, we investigate summability problems in higher-order weighted Dirichlet spaces $\widehat{\mathcal{H}}_{\mu,m}$ and show that Taylor partial sums are not uniformly bounded in these spaces and may therefore diverge in norm. To restore convergence, we introduce a family of modified polynomials whose coefficients are adjusted by a suitable weight array.
Under mild boundedness and variation assumptions on the weights, we establish norm convergence of the modified sums via a coefficient correspondence principle and a Local Douglas formula.
As an application, when the weight measure $\mu$ is a finite sum of Dirac point masses, explicit formulas for the modified coefficients are obtained, yielding a Fej\'er-type summability theorem for higher-order weighted Dirichlet spaces.

\end{abstract}

\maketitle
	\setcounter{section}{0}
\section{Introduction}

Let $\mathbb{D}$ denote the open unit disk in the complex plane and let $H(\mathbb{D})$ be the space of all analytic functions on $\mathbb{D}$.
For $f(z) = \sum_{k=0}^{\infty} a_k z^k \in H(\mathbb{D})$, denote by
\[
S_n f(z) = \sum_{k=0}^{n} a_k z^k
\]
the $n$-th Taylor partial sum of $f$.
A classical result of Fej\'er asserts that, in the Hardy space $H^2$, the Ces\`aro means of $S_n f$ converge to $f$ in norm.
This theorem plays a fundamental role in complex analysis and harmonic analysis.

Beyond the Hardy space, norm convergence of Taylor partial sums becomes more subtle.
In the Dirichlet space and its weighted variants, it is well known that Taylor partial sums may fail to converge in norm.
Various summability methods have therefore been introduced to recover convergence, including Ces\`aro and Abel-type means.
Such questions have attracted sustained attention in the study of function spaces and operator theory.

Let $\mathbb{N}$ denote the set of nonnegative integers and let $\mathbb{N}^+$ denote the set of positive integers, and let $\mathcal{M}_+(\mathbb{T})$ denote the set of all finite nonnegative Borel measures on the unit circle $\mathbb{T}$.

For $\mu \in \mathcal{M}_+(\mathbb{T})$, $f \in H(\mathbb{D})$, and $m \in \mathbb{N}^+$, the \emph{weighted Dirichlet-type integral of order $m$} is defined by
\[
D_{\mu,m}(f)
:= \frac{1}{m!(m - 1)!}
\int_{\mathbb{D}} \left| f^{(m)}(z) \right|^2 P_\mu(z)\,(1 - |z|^2)^{m - 1} \, dA(z),
\]
where $dA(z)$ denotes the normalized Lebesgue area measure on $\mathbb{D}$, $f^{(m)}$ is the $m$-th derivative of $f$, and $P_\mu$ is the Poisson integral of $\mu$, given by
\[
P_\mu(z) := \int_{\mathbb{T}} \frac{1 - |z|^2}{|z - \zeta|^2} \, d\mu(\zeta),
\quad z \in \mathbb{D}.
\]

For each $m \in \mathbb{N}$ and nonzero $\mu \in \mathcal{M}_+(\mathbb{T})$, we define
\[
\mathcal{H}_{\mu,m}
:= \left\{ f \in H(\mathbb{D}) : D_{\mu,m}(f) < \infty \right\},
\]
which is a semi-inner product space endowed with the seminorm $\sqrt{D_{\mu,m}(\cdot)}$.

If $\mu = \delta_\lambda$ is the Dirac measure at a point $\lambda \in \mathbb{T}$, the resulting space $\mathcal{H}_{\mu,m}$ is denoted by $\mathcal{H}_{\lambda,m}$ and is referred to as the \emph{local Dirichlet space of order $m$ at $\lambda$}.
If $\mu = 0$, then $\mathcal{H}_{\mu,m} = H^2$ for every $m \in \mathbb{N}$, where $H^2$ denotes the classical Hardy space on $\mathbb{D}$.

In the special case where $\mu = \sigma$, the normalized Lebesgue measure on $\mathbb{T}$, a direct computation shows that for any
$f(z) = \sum_{k = 0}^\infty a_k z^k \in H(\mathbb{D})$,
\[
D_{\sigma,m}(f) = \sum_{k = m}^{\infty} \binom{k}{m} |a_k|^2,
\quad m \in \mathbb{N},
\]
where the binomial coefficient is defined by
$\binom{k}{m} := \frac{k!}{m!(k - m)!}$ for $k \ge m$.

Associated with $\mathcal{H}_{\mu,m}$, we introduce the linear space
\[
\widehat{\mathcal{H}}_{\mu,m}
:= \left\{
f(z) = \sum_{j = m}^{\infty} a_j z^j : D_{\mu,m}(f) < \infty
\right\}.
\]
For any nonzero $\mu \in \mathcal{M}_+(\mathbb{T})$, the functional $D_{\mu,m}^{1/2}(\cdot)$ defines a seminorm on $\mathcal{H}_{\mu,m}$.
Moreover, $D_{\mu,m}(f) = 0$ if and only if $f$ is a polynomial of degree at most $m-1$.
In contrast, the same functional induces a norm on $\widehat{\mathcal{H}}_{\mu,m}$, which therefore becomes a Hilbert space.

Following earlier works, the higher-order weighted Dirichlet space
$\widehat{\mathcal{H}}_{\mu,m}$ introduced above
serves as a natural framework for studying higher-order differentiation
and composition operators. In recent years, norm convergence problems in weighted Dirichlet spaces have been extensively studied.
For the first-order case $m=1$, several authors have obtained positive summability results under appropriate assumptions on the weights.
In particular, Fricain and Mashreghi developed a local approach based on Douglas' formula, which has proved effective in treating weighted Dirichlet norms.
However, these methods can't be directly extended to the higher-order setting.
One major obstacle is that, for $m \ge 2$, Taylor partial sums generally fail to converge in the norm of $\widehat{\mathcal{H}}_{\mu,m}$, even for functions with strong regularity.

The first goal of this paper is to demonstrate this failure of convergence in a precise way.
We show that the sequence of Taylor partial sums is not uniformly bounded in $\widehat{\mathcal{H}}_{\mu,m}$, which implies, by the Uniform Boundedness Principle, the existence of functions whose partial sums diverge in norm.
This phenomenon highlights a fundamental difference between higher-order weighted Dirichlet spaces and their first-order counterparts.

Motivated by Fej\'er's theorem, our second and main objective is to construct a summability method that restores norm convergence.
To this end, we introduce a family of modified polynomials whose coefficients are adjusted by a carefully designed weight array.
Roughly speaking, the weights approximate the identity while controlling oscillation along each row.
Under mild boundedness and variation conditions, we prove that these modified sums converge in the $\widehat{\mathcal{H}}_{\mu,m}$-norm.

Our approach combines a coefficient correspondence principle with the Local Douglas formula.
This allows us to reduce global norm estimates to local ones and to obtain quantitative control on the modified coefficients.
As a concrete application, when the measure $\mu$ is a finite sum of Dirac point masses, we derive explicit formulas for the modified coefficients, yielding a constructive summability scheme.

To some extent, the results obtained in this paper provide a Fej\'er-type summability theorem in the setting of higher-order weighted Dirichlet spaces.

The paper is organized as follows. In Section~2, we recall basic properties of higher-order weighted Dirichlet spaces and establish preliminary estimates.
We also set an counterexample to show the failure of norm convergence of Taylor partial sums.
In Section~3, we introduce the modified summation method and prove the main convergence theorem.
Finally, in Section~4, we consider the case of Dirac measures and present explicit coefficient formulas.

\section{Bridging Formulae and Counterexamples}

Weighted Dirichlet spaces exhibit substantial deviations from classical polynomial approximation theory. Although the local Douglas formula provides a precise structural decomposition of functions in terms of boundary behavior and derivatives, and the coefficient correspondence theorem connects Taylor coefficients with weighted Dirichlet norms, the associated Taylor partial sums are in general not stable in higher-order weighted settings. This discrepancy is manifested through explicit counterexamples, revealing a fundamental obstruction to norm convergence of usual polynomial approximations. The purpose of this section is to clarify this phenomenon by combining structural identities with concrete divergence examples, thereby laying the groundwork for the summability methods developed in subsequent sections.

\subsection{Local Douglas formula for higher-order weighted Dirichlet-type integrals}

Our analysis relies on two fundamental tools.
The first is the local Douglas formula, which characterizes higher-order weighted Dirichlet spaces through factorization near boundary points.
The second is the coefficient correspondence theorem, which translates Dirichlet-type norms into weighted $\ell^2$ conditions on Taylor coefficients.
Together, these results reveal the structural origin of instability in polynomial approximations.

\begin{lemma}\label{lm2.1}
Let \( m \in \mathbb{N}^{+} \), \( \lambda \in \mathbb{T} \), and \( f \in H(\mathbb{D}) \).
Then \( f \in \mathcal{H}_{\lambda,m} \) if and only if
\[
f(z) = \alpha + (z - \lambda)g(z)
\]
for some \( \alpha \in \mathbb{C} \) and \( g \in \mathcal{H}_{\sigma,m-1} \).
Moreover, in this case the following statements hold:
\begin{itemize}
\item[(i)] \( D_{\lambda,m}(f) = D_{\sigma,m-1}(g) \);
\item[(ii)] \( f \) admits a non-tangential limit at \( \lambda \), given by
\[
f^*(\lambda) = \lim_{\substack{z \to \lambda \\ z \in \Omega_\lambda}} f(z) = \alpha,
\]
where the approach region
\[
\Omega_\lambda := \left\{ z \in \mathbb{D} : |z - \lambda| < \kappa (1 - |z|^2)^{1/2} \right\}
\]
is an oricyclic neighborhood of \( \lambda \) for arbitrary \( \kappa > 0 \).
\end{itemize}
\end{lemma}

\begin{proof}
A detailed proof can be found in \cite{MR4588160}.
\end{proof}

This lemma not only unifies several classical descriptions of weighted Dirichlet-type spaces, but also exposes a key obstruction: despite their natural algebraic form, Taylor polynomials do not necessarily respect the underlying norm structure.

\subsection{Failure of usual polynomial approximation}

\begin{lemma}\label{lm3.1}
Let $m,n \in \mathbb{N}^{+}$ and let
\[
f(z) = \sum_{k=m}^{\infty} a_k z^k \in \widehat{\mathcal{H}}_{1,m}.
\]
Then its Taylor partial sum
\[
S_{m,n}f(z) := \sum_{k=m}^{m+n} a_k z^k
\]
also belongs to $\widehat{\mathcal{H}}_{1,m}$.
\end{lemma}

\begin{proof}
By Lemma~\ref{lm2.1}, there exists a function
\[
g(z) = \sum_{k=0}^{\infty} b_k z^k \in \mathcal{H}_{\sigma,m-1}
\]
such that
\[
f(z) = a + (z-1)g(z).
\]
The coefficients satisfy
\[
\begin{cases}
b_0 = b_1 = \cdots = b_{m-1} = a,\\
b_{k-1} - b_k = a_k, \quad k \ge m.
\end{cases}
\]
A direct computation yields
\begin{align*}
S_{m,n}f(z)
&= a - b_{m+n} + (z-1)\sum_{k=0}^{m+n-1} (b_k - b_{m+n})z^k.
\end{align*}
Applying Lemma~\ref{lm2.1} again, we conclude that $S_{m,n}f \in \widehat{\mathcal{H}}_{1,m}$.
\end{proof}

\begin{theorem}\label{tm3.2}
Let \( m \in \mathbb{N} \) and let \( \mu \) be a positive finite Borel measure on \( \overline{\mathbb{D}} \).
Then there exists a function \( f \in \widehat{\mathcal{H}}_{\mu,m} \) such that
\[
D_{\mu,m}(S_{m,n}f - f) \not\to 0.
\]
\end{theorem}

\begin{proof}
The idea is to construct an explicit sequence for which the operator norms of the partial sum operators blow up, and then invoke the Uniform Boundedness Principle. Consider the linear map
\begin{align*}
S_{m,n}: &\widehat{\mathcal{H}}_{1,m} \longmapsto \widehat{\mathcal{H}}_{1,m},\\
          &\quad f \longmapsto S_{m,n}(f).
       \end{align*}
Let \( f_n(z) = n z^{n + m + 1} - (n + 1) z^{n + m} + z^m \) and $(S_{m,n}f_n)(z) = - (n + 1) z^{n + m} + z^m.$
Note that
\begin{align*}
f_n(z) &= n z^{n + m + 1} - (n + 1) z^{n + m} + z^m \\
       &= (z - 1) \left( n z^{n + m} - z^m - z^{m + 1} - \cdots - z^{m + n - 1} \right) \\
       &= (z-1){g_n}(z) \in \widehat{\mathcal{H}}_{1,m} \\
       \end{align*}
       and
\begin{align*}
(S_{m,n}f_n)(z)
&= - (n + 1) z^{n + m} + z^m \\
&= - n - (n + 1)(z - 1)(1 + z + \cdots + z^{m + n - 1}) \\
&\quad + (z - 1)(1 + z + \cdots + z^{m - 1}) \\
&= - n - (z - 1){v_n}(z) \in \widehat{\mathcal{H}}_{1,m},
\end{align*}
where
\begin{align*}
{g_n}(z)&:=n z^{n + m} - z^m - z^{m + 1} - \cdots - z^{m + n - 1}
\end{align*}
and
\begin{align*}
{v_n}(z)&:= (n + 1)(1 + z + \cdots + z^{m + n - 1})+ (1 + z + \cdots + z^{m - 1}) \\
&= n (1 + z + \cdots + z^{m - 1})+(n+1) (z^{m} + z^{m + 1} + \cdots + z^{m+n - 1}).
\end{align*}
Applying Lemma \ref{lm2.1}, we deduce that
\begin{align*}
D_{1,m}(S_{m,n}f_n)&= D_{\sigma ,m - 1} ({v_n}(z))\\
&= D_{\sigma ,m - 1} [ - n z^{m - 1} - (n + 1) \left( z^m + z^{m + 1} + \cdots + z^{m + n - 1} \right) ] \\
&= \sum_{j = m}^{n + m - 1} (n + 1)^2 \binom{j}{m - 1} + n^2 \binom{m - 1}{m - 1} \\
&= \sum_{j = m-1}^{n + m - 1} (n + 1)^2 \binom{j}{m - 1} - 2n - 1 \\
&= (n+1)^2 \binom{m+n}{m} - 2n - 1,
\end{align*}
and
\begin{align*}
D_{1,m}(f_n)&=D_{\sigma ,m - 1} ({g_n}(z))= D_{\sigma ,m - 1} \left( n z^{n + m} - z^m - \cdots - z^{n + m - 1} \right) \\
&= \sum_{j = m}^{m + n - 1} \binom{j}{m - 1} + n^2 \binom{m + n}{m - 1} \\
&= n^2 \binom{m + n}{m - 1} + \binom{m + n}{m} - 1.
\end{align*}
A straightforward computation yields that
\begin{align*}
\left\| S_{m,n} \right\|_{\widehat{\mathcal{H}}_{1,m} \to \widehat{\mathcal{H}}_{1,m}}
&\ge \frac{\left\| S_{m,n} f_n \right\|_{1,m}}{\left\| f_n \right\|_{1,m}}
= \frac{\sqrt{D_{\sigma,m-1}(v_n)}}{\sqrt{D_{\sigma,m-1}(g_n)}} \\
&= \sqrt{\frac{(n+1)^2 \binom{m+n}{m} - 2n - 1}{n^2 \binom{m+n}{m-1} + \binom{m+n}{m} - 1}}
\rightarrow \infty \quad \text{as } n \to \infty.
\end{align*}
By the Uniform Boundedness Principle, there exists a function
\( f \in \widehat{\mathcal{H}}_{\mu,m} \)
such that
\[
\sup_n \| S_{m,n} f \|_{\mu,m} = \infty.
\]
In particular,
\[
\| S_{m,n}f - f \|_{\mu,m} \not\to 0,
\]
which completes the proof.
\end{proof}

The above counterexample demonstrates that, even though Taylor polynomials remain admissible elements of weighted Dirichlet spaces, their partial sums may fail to approximate functions in norm.
This instability motivates the introduction of modified summability methods, which will be developed in the next section.

\section{Convergence via Weighted Polynomial Modifications}

The counterexamples in Section~2 show that usual Taylor partial sums fail to approximate functions in
$\widehat{\mathcal{H}}_{\mu,m}$.
In this section, we demonstrate that this instability can be overcome by introducing suitable modifications to the Taylor coefficients.
More precisely, we construct a class of \emph{weighted modified polynomials} whose coefficients are adjusted by a carefully controlled array of weights.
These modifications preserve the algebraic simplicity of polynomial truncations while restoring norm convergence in higher-order weighted Dirichlet spaces.

Our main result asserts that for any $f \in \widehat{\mathcal{H}}_{\mu,m+1}$, there exists a sequence of modified Taylor polynomials $\{p_n\}$ such that
\[
\|f - p_n\|_{\mu,m} \longrightarrow 0 \quad \text{as } n \to \infty.
\]
The proof relies on the local Douglas formula, a refined coefficient correspondence, and a delicate balance between coefficient decay and oscillation control.

To implement this approximation scheme, we introduce an array of complex weights
$\bigl(w_{n,k}\bigr)_{n\ge0,\;k\ge m+1}$,
which governs the modification of Taylor coefficients.

\begin{definition}\label{def4.1}
Let \( (w_{n,k})_{n \ge 0,\, k \ge m+1} \) be an array of complex numbers satisfying:
\begin{align}
&(1)\; w_{n,k}=0, && k>m+n+1; \tag{3.1}\\
&(2)\; |w_{n,k}|\le M, && m+1\le k\le m+n+1; \tag{3.2}\\
&(3)\; \lim_{n\to\infty} w_{n,k}=1, && \text{for each fixed } k\ge m+1; \tag{3.3}\\
&(4)\; |w_{n,k+1}-w_{n,k}|
\le \frac{L}{\sqrt{(n+1)\binom{m+n+1}{m}}}, && m+1\le k<m+n+1. \tag{3.4}
\end{align}
\end{definition}

The array $(w_{n,k})$ can be viewed as a bounded lower-triangular matrix

\[
\begin{pmatrix}
w_{0,m+1} & 0 & 0 & 0 & \cdots & \cdots\\
w_{1,m+1} & w_{1,m+2} & 0 & 0 & \cdots & \cdots\\
w_{2,m+1} & w_{2,m+2} & w_{2,m+3} & 0 & \cdots & \cdots\\
\vdots & \vdots & \vdots & \ddots & \vdots & \cdots\\
w_{k,m+1} & w_{k,m+2} & w_{k,m+3} & \cdots & w_{k,m+k+1} & \cdots\\
\cdots & \cdots & \cdots & \cdots & \cdots &\cdots
\end{pmatrix}
.\]

Condition~(3.3) ensures asymptotic consistency with classical Taylor truncations,
while condition~(3.4) controls oscillations between adjacent coefficients.
Together, these assumptions guarantee stability of the modified polynomials in weighted Dirichlet norms.

\begin{lemma}\label{lm4.2}
Let
\[
f(z)=\sum_{k=m+1}^{\infty} a_k z^k \in \widehat{\mathcal{H}}_{\lambda,m+1}
\]
with representation $f(z)=a+(z-\lambda)g(z)$,
where $a\in\mathbb{C}$ and $g(z)=\sum_{k=0}^{\infty} b_k z^k \in \mathcal{H}_{\sigma,m}$.
Then the coefficients satisfy
\[
\begin{cases}
a=\lambda b_0,\\
b_{k-1}=\lambda b_k, & 1\le k\le m,\\
b_{k-1}-\lambda b_k=a_k, & k\ge m+1.
\end{cases}
\]
\end{lemma}
\begin{proof}
The result follows from a direct comparison of coefficients after expanding
$f(z)=a+(z-\lambda)g(z)$.
\end{proof}

\begin{definition}\label{def4.3}
Let
\[
f(z)=\sum_{k=m+1}^{\infty} a_k z^k \in \widehat{\mathcal{H}}_{\lambda,m+1}.
\]
The \emph{modified Taylor polynomial} of order $n$ is defined by
\[
p_n(z)=\sum_{k=m+1}^{m+n+1} w_{n,k}\, a_k z^k,
\]
where $(w_{n,k})$ satisfies Definition~\ref{def4.1}.
\end{definition}

Unlike the usual partial sums, the sequence $\{p_n\}$ incorporates coefficient smoothing
and will be shown to converge in weighted Dirichlet norms. The following inclusion property allows us to reduce the approximation problem
to a lower-order weighted Dirichlet space.

\begin{lemma}\label{lm4.4}
If $\lambda \in \mathbb{T}$ and $m \in {N^ + }$, then $ \widehat{\mathcal{H}}_{\lambda, m+1 } \subseteq \widehat{\mathcal{H}}_{\lambda, m }$.
\end{lemma}

\begin{proof}
If
\begin{equation*}
f(z) = \sum\limits_{k = m+1}^\infty  a_k z^k \in \widehat{\mathcal{H}}_{\lambda ,m+1},
\end{equation*}
then
\begin{equation*}
\sum\limits_{k = m+1}^\infty a_k z^k = a + (z - \lambda) \sum\limits_{k = 0}^\infty b_k z^k.
\end{equation*}
and
\begin{equation*}
\sum\limits_{k = 0}^\infty b_k z^k \in \mathcal{H}_{\sigma ,m}.
\end{equation*}
That is, $\sum\limits_{k = m}^{\infty} \left| b_k \right|^2 \binom{k}{m} < \infty $, which implies that
\begin{align*}
\sum_{k = m - 1}^{\infty} \left| b_k \right|^2 \binom{k}{m - 1}&= \left| b_{m - 1} \right|^2 + \sum_{k = m}^{\infty} \left| b_k \right|^2 \binom{k}{m - 1} \\
&= \left| b_{m - 1} \right|^2 + m\sum_{k = m}^{\infty} \frac{{{{\left| {{b_k}} \right|}^2}}}{{k - m + 1}} \binom{k}{m} < \infty.
\end{align*}
This means $\sum\limits_{k = 0}^\infty b_k z^k \in \mathcal{H}_{\sigma ,m-1}$, and $f(z) = \sum\limits_{k = m+1}^\infty  a_k z^k \in \widehat{\mathcal{H}}_{\lambda ,m}.$ The proof is completed.
\end{proof}

As a first step, we establish norm convergence for the modified polynomials at a fixed boundary point $\lambda\in\mathbb{T}$.

\begin{corollary}\label{co4.5}
Given $f(z) = \sum\limits_{k = m + 1}^\infty  {{a_k}{z^k}} \in \widehat{\mathcal{H}}_{\lambda, m+1}$ with the representation $f(z) = a + (z - \lambda) g(z), \text{~where~}  a\in\mathbb{C}$ and $g(z) = \sum\limits_{k=0}^{\infty} b_k z^k \in \mathcal{H}_{\sigma, m}$. Define
\begin{equation}
 {g_n}(z): = \sum\limits_{k = m}^{n + m} {{w_{n,k + 1}}} {b_k}{z^k} \tag{3.5}
 \end{equation}
 and set $ f_n(z) = (z - \lambda) g_n(z) $, then
$${D_{\lambda ,m}}(f - {f_n}) \to 0 \quad as \quad n \to \infty.$$
\end{corollary}

\begin{proof}
Since Lemma \ref{lm4.4} gives $f(z) = \sum\limits_{k = m+1}^{\infty}  a_k z^k \in \widehat{\mathcal{H}}_{\lambda ,m+1} \subseteq \widehat{\mathcal{H}}_{\lambda ,m}$,  we obtain $g \in \mathcal{H}_{\sigma ,m-1}$ and ${g_n} \in \widehat{\mathcal{H}}_{\sigma ,m-1}$ by Lemma \ref{lm2.1}, which implies that ${D_{\lambda ,m}}(f - {f_n}) = {D_{\sigma ,m - 1}}(g - {g_n})$.
For arbitrary $ \varepsilon >0 $, we can choose a sufficiently large $N\in\mathbb{N}$ such that
\begin{equation}
 \sum\limits_{k = N}^\infty  {|{b_k}{|^2}} \binom{k}{m - 1}<{\varepsilon ^2}. \tag{3.6}
 \end{equation}
Then
\begin{align*}
&D_{\sigma ,m - 1}^{1/2}\left( {g - {g_n}} \right) = D_{\sigma ,m - 1}^{1/2}\left( {\sum\limits_{k = 0}^\infty  {{b_k}} {z^k} - \sum\limits_{k = m}^{n + m} {{w_{n,k + 1}}} {b_k}{z^k}} \right)\\
&=D_{\sigma ,m - 1}^{1/2}\left( {{b_{m - 1}}{z^{m - 1}} + \sum\limits_{k = m}^{N - 1} {{b_k}{z^k}}  - \sum\limits_{k = m}^{N - 1} {{w_{n,k + 1}}{b_k}{z^k} + \sum\limits_{k = N}^\infty  {{b_k}{z^k} - \sum\limits_{k = N}^{n + m} {{w_{n,k + 1}}{b_k}{z^k}} } } } \right)\\
& \le D_{\sigma ,m - 1}^{1/2}\left[ {\sum\limits_{k = m}^{N - 1} {(1 - {w_{n,k + 1}}){b_k}{z^k}} } \right] + D_{\sigma ,m - 1}^{1/2}\left( {\sum\limits_{k = N}^\infty  {{b_k}{z^k}} } \right) + D_{\sigma ,m - 1}^{1/2}\left( {\sum\limits_{k = N}^{n + m} {{w_{n,k + 1}}{b_k}{z^k}} }, \right)
\end{align*}
where by \((3.3)\), the first term tends to 0 as $n \to \infty$ and by \((3.6)\), the second term and the third term tend to 0 as $n \to \infty$.

Hence, ${D_{\lambda ,m}}(f - {f_n})={D_{\sigma ,m-1}}(g - {g_n}) \to 0 $ as $n \to \infty$.
\end{proof}

\begin{corollary}\label{co4.6}
Let \( f(z) = a + (z - \lambda)g(z) \in \widehat{\mathcal{H}}_{\lambda, m+1} \) with $ a\in\mathbb{C} $ and $g(z) = \sum\limits_{k = 0}^\infty  {{b_k}{z^k}}  \in H_{\sigma, m} $. We define ${g_n}(z)$ as in (3.5) and ${f_n}(z) := (z - \lambda ){g_n}(z)$. Then there exists a positive constant ${C_1}$ such that
$${D_{\lambda ,m}}(f - {f_n}) \le C_1^2{D_{\lambda ,m+1}}(f).$$
\end{corollary}

\begin{proof}
\begin{align*}
&{D_{\lambda ,m}}\left( {f - {f_n}} \right) = {D_{\sigma ,m - 1}}\left( {g - {g_n}} \right)\\
& \le 2{D_{\sigma ,m - 1}}\left( g \right) + 2{D_{\sigma ,m - 1}}\left( {{g_n}} \right)\\
&= 2\sum\limits_{k = m - 1}^\infty  {{{\left| {{b_k}} \right|}^2}}\binom{k}{m-1} + 2\sum\limits_{k = m}^{n + m} {{{\left| {{w_{n,k + 1}}} \right|}^2}{{\left| {{b_k}} \right|}^2}}\binom{k}{m-1} \\
& \le C_1^2{D_{\sigma ,m}}(g)= C_1^2{D_{\lambda ,m+1}}(f).
\end{align*}
\end{proof}

 \begin{lemma}\label{lm4.7}
Let \( q(z) = \sum\limits_{k = m+1}^{n + m +1} c_k z^k \in \widehat{\mathcal{H}}_{\sigma, m} \). Then
\[
D_{\lambda ,m}(q) \le (n + 1) \binom{n+m+1}{m} D_{\sigma ,m}(q).
\]
\end{lemma}
\begin{proof}
Writing \( q(z) \) in terms of its decomposition around the point \( \lambda \in \mathbb{T} \) as in Lemma \ref{lm2.1}, we see that
\begin{align*}
\sum_{k = m + 1}^{n + m + 1} c_k z^k
&= q(\lambda) + (z - \lambda) \sum_{k = 0}^{n + m} d_k z^k \\
&= \left[ {q(\lambda ) - \lambda {d_0}} \right] + d_{n+m} z^{n+m+1} + \sum_{k = 1}^{n+m} (d_{k-1} - \lambda d_k) z^k.
\end{align*}

By comparing coefficients, we obtain
\[
\left\{
\begin{array}{ll}
q(\lambda ) = \lambda {d_0},\\
{d_{k - 1}} = \lambda {d_k}, & \text{for } 1 \le k \le m, \\
c_k = d_{k-1} - \lambda d_k, & \text{for } m+1 \le k \le n + m, \\
c_{n+m+1} = d_{n+m}. &
\end{array}
\right.
\]

Now, we estimate the coefficients $d_{k}$ in three cases.

For \( m \le k \le n + m - 1 \), we can express each \( d_k \) in terms of the \( c_j \)
\[
d_k = \sum_{j = k + 1}^{n + m + 1} \lambda^{j - (k + 1)} c_j.
\]
This implies
\[
|d_k|^2 \le \left( \sum_{j = k + 1}^{n + m + 1} |\lambda|^{2(j - k - 1)} \right)
\left( \sum_{j = k + 1}^{n + m + 1} |c_j|^2 \right)
= (n + m +1 - k) \sum_{j = k + 1}^{n + m + 1} |c_j|^2
\]
by using Cauchy-Schwarz inequality.

For \( m \le k \le n + m - 1 \), we have
\begin{align*}
|d_k|^2
&\le (n + 1) \sum_{j = m + 1}^{n + m + 1} |c_j|^2 \binom{j}{m}
= (n + 1) D_{\sigma,m}(q).
\end{align*}
For $ k=m-1 $ or  $ k=n+m $, similar estimation can be obtained that

\[
|d_{m-1}|^2 = |d_{m}|^2 \le (n + 1) D_{\sigma, m}(q)
\]
and
\[
|d_{n+m}|^2 = |c_{n+m+1}|^2 \le (n + 1) D_{\sigma, m}(q).
\]

By using Lemma \ref{lm2.1} again, we get
\begin{align*}
D_{\lambda, m}(q)
&= D_{\sigma, m - 1} \left( \sum_{k = 0}^{n + m} d_k z^k \right)
= \sum_{k = m - 1}^{n + m} |d_k|^2 \binom{k}{m - 1} \\
&\le \sum_{k = m - 1}^{n + m} (n + 1) D_{\sigma, m}(q) \binom{k}{m - 1}
= (n + 1) \binom{n + m + 1}{m} D_{\sigma, m}(q).
\end{align*}
This completes the proof.
\end{proof}

We now extend the pointwise convergence results to the setting of finite positive
Borel measures on $\mathbb{T}$.
The next lemma establishes almost-everywhere convergence of the modified polynomials,
together with a uniform domination estimate, which is essential for applying the
dominated convergence theorem.

\begin{lemma}\label{lm4.8}
Let \( \mu \) be a finite positive Borel measure on \( \mathbb{T} \), \(  f(z)=\sum\limits_{k = m+1}^\infty  {{a_k}{z^k}} \in \widehat{\mathcal{H}}_{\mu, m+1} \) and \( p_n(z) = \sum\limits_{k = m+1}^{n + m + 1} w_{n,k} a_k z^k \). Then
\[
\int_{\mathbb{T}} D_{\lambda,m+1}(f) \, d\mu(\lambda) < \infty,
\]
and in particular, \( D_{\lambda,m+1}(f) < \infty \) for \( \mu \)-almost every \(\lambda \in \mathbb{T} \). Furthermore, for each such \( \lambda \), we have
\[
\lim_{n \to \infty} D_{\lambda,m}(f - p_n) = 0 \quad \text{and} \quad
D_{\lambda,m}(f - p_n) \le C^2 D_{\lambda,m+1}(f),
\]
where \( C \) is a positive constant depending only on the array \( (w_{n,k})_{n \geq 0, k \geq m+1} \).
\end{lemma}

\begin{proof}
For \(  f(z)=\sum\limits_{k = m+1}^\infty  {{a_k}{z^k}} \in \widehat{\mathcal{H}}_{\mu, m+1} \), there exists some \( \alpha \in \mathbb{C}\) and \( g(z) = \sum\limits_{k = 0}^\infty  {{b_k}{z^k}} \in {\mathcal{H}_{\sigma ,m}} \) such that \( f(z) = \alpha + (z - \lambda)g(z)\). Let \( {g_n}(z) = \sum\limits_{k = m}^{n + m} {{w_{n,k + 1}}{b_k}{z^k}}\) and \({f_n}(z) = (z - \lambda ){g_n}(z)\), where \( (w_{n,k})_{n \geq 0, k \geq m+1} \) is an array defined in Definition \ref{def4.1}. For arbitrary \( \varepsilon > 0 \) , choose a sufficiently large \( N \in \mathbb{N} \) such that
\[
\sum_{k = m + N +1}^{\infty} |b_k|^2 \binom{k}{m} < \varepsilon^2. \tag{3.7}
\]
By the triangle inequality, we have
\begin{align*}
D_{\lambda ,m}^{1/2}(p_n - f)
&= D_{\lambda ,m}^{1/2}\left[(p_n - f_n) + (f_n - f)\right] \\
&\le D_{\lambda ,m}^{1/2}(p_n - f_n) + D_{\lambda ,m}^{1/2}(f_n - f).
\end{align*}
A direct computation yields that
\begin{align*}
f_n(z)
&= (z - \lambda) g_n(z) = (z - \lambda) \sum_{k = m}^{n + m} w_{n,k + 1} b_k z^k \\
&= \sum_{k = m+1}^{n + m + 1} w_{n,k} b_{k - 1} z^k
- \lambda \sum_{k = m}^{n + m} w_{n,k + 1} b_k z^k \\
&= \sum_{k = m+1}^{n + m + 1} w_{n,k} (b_{k - 1} - \lambda b_k) z^k
+ \lambda \sum_{k = m+1}^{n + m + 1} (w_{n,k} - w_{n,k + 1}) b_k z^k - \lambda {w_{n,m+1}}{b_{m}}{z^{m}}.
\end{align*}
By the relationship give in Lemma \ref{lm4.2}, \( a_k \) = \( b_{k-1} - \lambda b_k \), which implies that
\[
f_n(z) = \sum_{k = m + 1}^{n + m + 1} w_{n,k} a_k z^k
+ \lambda \sum_{k = m + 1}^{n + m + 1} (w_{n,k} - w_{n,k + 1}) b_k z^k - \lambda {w_{n,m+1}}{b_{m}}{z^{m}}.
\]
Therefore,
\[
f_n(z) - p_n(z) = \lambda \sum_{k = m + 1}^{n + m + 1} (w_{n,k} - w_{n,k + 1}) b_k z^k - \lambda {w_{n,m+1}}{b_{m}}{z^{m}},
\]
and then
\begin{align*}
&D_{\lambda ,m}^{1/2}(f_n - p_n)
\le D_{\lambda ,m}^{1/2} \left[ \sum_{k = m + 1}^{n + m + 1} (w_{n,k} - w_{n,k + 1}) b_k z^k \right] + D_{\lambda ,m}^{1/2}({w_{n,m+1}}{b_{m}}{z^{m}})\\
&\le D_{\lambda ,m}^{1/2} \left[ \sum_{k = m + 1}^{m + N} (w_{n,k} - w_{n,k + 1}) b_k z^k \right]
+ D_{\lambda ,m}^{1/2} \left[ \sum_{k = m + N + 1}^{n + m} (w_{n,k} - w_{n,k + 1}) b_k z^k \right]\\
&\quad+ D_{\lambda ,m}^{1/2} \left[(w_{n,n + m + 1} - w_{n,n + m + 2}) b_{n + m + 1} z^{ n + m + 1} \right]\\
&\le \left[
N \binom{N + m}{m} \sum_{k = m + 1}^{m + N}
|w_{n,k + 1} - w_{n,k}|^2 |b_k|^2 \binom{k}{m}
\right]^{1/2} \\
&\quad +
\left[
(n + 1) \binom{n + m + 1}{m} \sum_{k = m + N + 1}^{n + m}
|w_{n,k} - w_{n,k + 1}|^2 |b_k|^2 \binom{k}{m}
\right]^{1/2}\\
&\quad + D_{\lambda ,m}^{1/2}\left( {{w_{n,n + m + 1}}{b_{n + m + 1}}{z^{n + m + 1}}} \right).
\end{align*}
By \((3.3)\), the first term tends to 0 as \( n \to \infty \). Moreover, it follows from \((3.4)\) and \((3.7)\) that the second term tends to 0 as \( n \to \infty \). It remains to estimate the third term. To this end, we first note that

\[
{w_{n,n + m + 1}}{b_{n + m + 1}}{z^{n + m + 1}}
= {\lambda ^{n + m + 1}}{w_{n,n + m + 1}}{b_{n + m + 1}} + (z - \lambda )\sum\limits_{k = 0}^{n + m} {{\lambda ^{(n + m) - k}}{w_{n,n + m + 1}}{b_{n + m + 1}}{z^k}}.
\]
This together with Lemma \ref{lm2.1}, yields
\begin{align*}
&D_{\lambda ,m}^{1/2}\left( {{w_{n,n + m + 1}}{b_{n + m + 1}}{z^{n + m + 1}}} \right) \\
&= D_{\sigma,m-1}^{1/2} \left( \sum_{k=0}^{n+m} \lambda^{(n+m)-k} w_{n,n+m+1} b_{n+m+1} z^k \right) \\
&= \left[ \sum_{k=m-1}^{n+m} |w_{n,n+m+1}|^2 |b_{n+m+1}|^2 \binom{k}{m-1} \right]^{1/2} \\
&= \left[ |w_{n,n+m+1}|^2 |b_{n+m+1}|^2 \binom{n+m+1}{m} \right]^{1/2}\\
& \leq M  \left[ \left| b_{n+m+1} \right|^2 \binom{n+m+1}{m} \right]^{1/2}.
\end{align*}
By \((3.7)\), $M  \left[ \left| b_{n+m+1} \right|^2 \binom{n+m+1}{m} \right]^{1/2} \to 0 $ as  $ n \to \infty$. This means that $D_{\lambda ,m}^{1/2}({f_n} - {p_n}) \to 0$ as $n \to \infty$. Since we have already known that
\[
D_{\lambda ,m}^{1/2}(f_n - f) \to 0 \ \text{as } n \to \infty,
\]
which indicates that
\[
\lim_{n \to \infty} D_{\lambda,m}(f - p_n) = 0.
\]

We still need to show the latter part of the estimate. By using Lemma \ref{lm4.7} and \((3.4)\), there exists some positive constant ${C_2}$ such that
\begin{align*}
&D_{\lambda ,m}(f_n - p_n)\le 2{D_{\lambda ,m}}\left[ {\sum\limits_{k = m+1}^{n + m} {({w_{n,k}} - {w_{n,k + 1}})} {b_k}{z^k}} \right] + 2{D_{\lambda ,m}}\left[ {{w_{n,n + m + 1}}{b_{n + m +1}}{z^{n + m + 1}}} \right] \\
&\le 2(n + 1) \binom{n+m+1}{m} D_{\sigma ,m} \left[ \sum_{k = m+1}^{n + m} (w_{n,k} - w_{n,k + 1}) b_k z^k \right] + 2{D_{\lambda ,m}}\left[ {{w_{n,n + m + 1}}{b_{n + m + 1}}{z^{n + m + 1}}} \right] \\
&= 2(n + 1) \binom{n+m+1}{m} \sum_{k = m+1}^{n + m}
|w_{n,k} - w_{n,k+1}|^2 |b_k|^2 \binom{k}{m} + 2\binom{n+m+1}{m}|w_{n,n+m+1} b_{n+m+1}|^{2}\\
&\le 2L^2 D_{\sigma ,m}(g) + 2\binom{n+m+1}{m}|w_{n,n+m+1} b_{n+m+1}|^{2}\\
&\le C_2^2{D_{\sigma ,m}}(g) = C_2^2{D_{\lambda ,m+1}}(f). \tag{3.8}
\end{align*}
Combining this with Corollary \ref{co4.6}, we see that there exists some positive constant ${C_1}$ such that
\begin{equation}
D_{\lambda ,m}(f - f_n) \le  C_1^2 D_{\lambda ,m+1}(f). \tag{3.9}
\end{equation}
Hence,
\[
D_{\lambda,m}(f - p_n) \le C^2 D_{\lambda,m+1}(f), \tag{3.10}
\]
where $C = {C_1} + {C_2}$. The proof is completed.
\end{proof}

\begin{theorem}\label{tm4.9}
Let $(w_{n,k})_{n\ge0,\;k\ge m+1}$ satisfy Definition~\ref{def4.1}.
For any
\[
f(z)=\sum_{k=m+1}^{\infty} a_k z^k \in \widehat{\mathcal{H}}_{\mu,m+1},
\]
define
\[
p_n(z)=\sum_{k=m+1}^{m+n+1} w_{n,k} a_k z^k.
\]
Then
\[
\|f-p_n\|_{\mu,m} \longrightarrow 0
\quad \text{as } n\to\infty.
\]
\end{theorem}

\begin{proof}
The conclusion follows directly from Lemma~\ref{lm4.8} by an application of the dominated convergence theorem.
\end{proof}

This result shows that although standard Taylor partial sums may diverge in
higher-order weighted Dirichlet spaces, appropriately modified polynomial
approximations recover norm convergence.
The construction provides a flexible and robust framework for polynomial
approximation beyond the classical setting.

\section{Dirac Measure Combinations and Closed-Form Solutions}

The results obtained in the previous sections are formulated for general finite positive Borel measures on $\mathbb{T}$.
In the special case where the measure $\mu$ is a finite linear combination of Dirac masses, the abstract approximation scheme developed earlier admits an explicit and computable realization. This allows us to close the abstract theory developed in the previous sections with a fully explicit summability scheme.

The purpose of this section is to derive closed-form modified Taylor polynomials in this setting and to establish their norm convergence in $\widehat{\mathcal{H}}_{\mu,m}$.

\subsection*{Recall of the Weight Array}

Throughout this section, we work with the same weight array
$(w_{n,k})_{n\ge0,\,k\ge m}$ introduced in Definition~3.1,
which satisfies conditions (3.1)--(3.4) therein.
For the reader's convenience, we recall that $w_{n,k}=0$ for $k>m+n$
and that the modification of Taylor coefficients occurs only in a finite range. For intuition, the array $(w_{n,k})$ can be viewed as a bounded lower-triangular matrix, encoding a finite-range modification of Taylor coefficients.

The following lemma shows that the coefficient modification induced by
$(w_{n,k})$ preserves norm convergence in the classical weighted Dirichlet space.

\begin{lemma}\label{lm5.2}
Let \( (w_{n,k})_{n \ge 0,\, k \ge m} \) satisfy \((3.1)\)--\((3.4)\).
For \( g(z)=\sum_{k=0}^{\infty} b_k z^k \in \mathcal{H}_{\sigma,m-1} \), define
\[
g_n(z):=\sum_{k=m-1}^{n+m-1} w_{n,k+1}\, b_k z^k.
\]
Then
\[
D_{\sigma,m-1}^{1/2}(g_n-g)\to 0
\quad\text{as } n\to\infty.
\tag{4.1}
\]
\end{lemma}

\begin{proof}
  Given $ \varepsilon >0 $, we choose a sufficiently large $N\in\mathbb{N}$ such that $\sum\limits_{k = N}^\infty  {|{b_k}{|^2}} \binom{k}{m - 1}<{\varepsilon ^2}$. A direct computation yields
  \begin{align*}
  D_{\sigma ,m - 1}(g_n)
  &= D_{\sigma ,m - 1}\left( \sum_{k = m - 1}^{n + m - 1} w_{n,k + 1} b_k z^k \right) \\
  &= \sum_{k = m - 1}^{n + m - 1} |w_{n,k + 1}|^2\, |b_k|^2 \binom{k}{m - 1}\\
  & \le {M^2}\sum_{k = m - 1}^{n + m - 1} |b_k|^2 \binom{k}{m - 1} \le {M^2}{D_{\sigma ,m - 1}}(g) < \infty ,
  \end{align*}
  and
  \begin{align*}
  {D_{\sigma ,m - 1}}({g_n} - g)
  &={D_{\sigma ,m - 1}}\left( {\sum\limits_{k = m - 1}^\infty  {{b_k}} {z^k} - \sum\limits_{k = m - 1}^{n + m - 1} {{w_{n,k + 1}}} {b_k}{z^k}} \right)\\
  &={D_{\sigma ,m - 1}}\left[ {\sum\limits_{k = m - 1}^{N - 1} {{b_k}(1 - {w_{n,k + 1}}){z^k} + \sum\limits_{k = N}^\infty  {{b_k}{z^k} - \sum\limits_{k = N}^{n + m - 1} {{w_{n,k + 1}}{b_k}{z^k}} } } } \right]\\
  &\le 2{D_{\sigma ,m - 1}}\left[ {\sum\limits_{k = m - 1}^{N - 1} {{b_k}} (1 - {w_{n,k + 1}}){z^k}} \right]
  + 4 D_{\sigma ,m - 1} \left( \sum_{k = N}^{\infty} b_k z^k \right) \\
  &\quad + 4 D_{\sigma ,m - 1} \left( \sum_{k = N}^{n + m - 1} w_{n,k + 1} b_k z^k \right)\\
  &\le 2{D_{\sigma ,m - 1}}\left[ {\sum\limits_{k = m - 1}^{N - 1} {{b_k}(1 - {w_{n,k + 1}}){z^k}} } \right] + 4(1 + M)^2{\varepsilon ^2}.
  \end{align*}
By \((3.3)\), ${D_{\sigma ,m - 1}}\left[ {\sum\limits_{k = m - 1}^{N - 1} {{b_k}(1 - {w_{n,k + 1}}){z^k}} } \right]$ tends to zero as $ n \to \infty$. Therefore, $$\limsup\limits_{n \to \infty} D_{\sigma, m - 1}^{1/2}(g_n - g) \le 2(1+M){\varepsilon}.$$
Since $\varepsilon>0$ is arbitrary, the conclusion follows.
\end{proof}

\begin{theorem}\label{tm5.3}
For any \( f \in \widehat{\mathcal{H}}_{\lambda,m} \),
there exists a sequence \( \{f_n\} \subset \mathcal{H}_{\lambda,m} \)
such that
\[
D_{\lambda,m}(f-f_n)\to 0
\quad\text{as } n\to\infty.
\]
\end{theorem}

\begin{proof}
Using the decomposition provided by Lemma~\ref{lm2.1},
we reduce the approximation problem in $\widehat{\mathcal{H}}_{\lambda,m}$
to one in $\mathcal{H}_{\sigma,m-1}$. By Lemma \ref{lm2.1}, for \( f \in \widehat{\mathcal{H}}_{\lambda, m} \), it can be decomposed as \( f(z)= a + (z - \lambda) g(z)\), where $a\in\mathbb{C}$ and \( g(z) = \sum\limits_{k = 0}^\infty  {{b_k}{z^k}}  \in \mathcal{H}_{\sigma, m - 1} \). Define ${g_n}(z): = \sum\limits_{k = m - 1}^{n + m - 1} {{w_{n,k + 1}}} {b_k}{z^k}{\rm{ }}$, and set $ f_n(z) = (z - \lambda) g_n(z) $. Lemma \ref{lm2.1} indicates that ${D_{\lambda ,m}}(f - {f_n}) = {D_{\sigma ,m - 1}}(g - {g_n})$. Combining this with Lemma \ref{lm5.2}, we see that $D_{\lambda, m}(f - f_n) \to 0 \text{ as }  n \to \infty$.
\end{proof}

\begin{remark}
Theorem~\ref{tm5.3} shows that approximation in the local higher-order
Dirichlet space reduces entirely to the classical weighted Dirichlet space,
highlighting the structural role of the local Douglas formula.
\end{remark}

Although the approximants $f_n$ constructed above do not belong to
$\widehat{\mathcal{H}}_{\lambda,m}$, the following corollary shows that
a further explicit modification restores this property.

\begin{corollary}\label{co5.4}
Let \( f(z)=\sum_{k=m}^{\infty} a_k z^k \in \widehat{\mathcal{H}}_{\lambda,m} \).
Define
\[
q_n(z)
=\sum_{k=m}^{n+m} a_k z^k
+\left(\sum_{k=n+m}^{\infty} a_k \lambda^{k-n-m}\right) z^{n+m}.
\]
Then \( q_n \in \widehat{\mathcal{H}}_{\lambda,m} \) and
\[
D_{\lambda,m}(f-q_n)\to 0
\quad\text{as } n\to\infty.
\]
\end{corollary}
\begin{proof}
Let ${w_{n,k}} = 1$ for $m \le k \le m + n$ and zero elsewhere. It is obvious that
\[{f_n}(z) = \sum\limits_{k = m}^{n + m} {{a_k}{z^k}}  + \lambda {b_{n + m}}{z^{n + m}} - \lambda {b_{m - 1}}{z^{m - 1}} \in \mathcal{H}_{\lambda, m}\]
Moreover, it follows from Theorem \ref{tm5.3} that ${D_{\lambda ,m}}(f - {f_n}) \to 0$ as $n \to \infty$. Set \[{q_n}(z) = \sum\limits_{k = m}^{n + m} {{a_k}{z^k}}  + \lambda {b_{n + m}}{z^{n + m}}.\]
Then ${q_n}(z) \in \widehat{\mathcal{H}}_{\lambda, m}$ and
\begin{align*}
q_n(z) &= \sum\limits_{k = m}^{n + m} a_k z^k  + \lambda b_{n + m} z^{n + m} \\
&= \sum\limits_{k = m}^{n + m - 1} a_k z^k + a_{n + m} z^{n + m} + \lambda b_{n + m} z^{n + m} \\
&= \sum\limits_{k = m}^{n + m - 1} a_k z^k  + b_{n + m - 1} z^{n + m}, \\
\end{align*}
where the third equality is from Lemma \ref{lm4.2}. Notice that for arbitrary $N \ge m + n$,
\begin{align*}
\sum_{k = n + m}^{N} a_k \lambda^{k - n - m} &= {\lambda ^{1 - n - m}}\sum\limits_{k = n + m}^N {{a_k}{\lambda ^{k - 1}}} \\
&={\lambda ^{1 - n - m}}\sum\limits_{k = n + m}^N {({b_{k - 1}} - \lambda {b_k}){\lambda ^{k - 1}}} \\
&={\lambda ^{1 - n - m}}\sum\limits_{k = n + m}^N {\left( {{b_{k - 1}}{\lambda ^{k - 1}} - {b_k}{\lambda ^k}} \right)} \\
&={\lambda ^{1 - n - m}}({b_{n + m - 1}}{\lambda ^{n + m - 1}} - {b_N}{\lambda ^N})\\
&={b_{n + m - 1}} - {b_N}{\lambda ^{N + 1 - n - m}},
\end{align*}
where ${b_N} \to 0$ as $n \to \infty$ since the sequence $\{ {b_k}\} $ satisfies $\displaystyle \sum_{k = m - 1}^{\infty} |b_k|^2 \binom{k}{m - 1} < \infty$.
Hence, $\sum\limits_{k = n + m}^\infty  {{a_k}{\lambda ^{k - n - m}}}  = {b_{n + m - 1}}$, and
\[{q_n}(z) = \sum\limits_{k = m}^{n + m - 1} {{a_k}} {z^k} + \left( {\sum\limits_{k = n + m}^\infty  {{a_k}} {\lambda ^{k - n - m}}} \right){z^{n + m}}{\rm{ }}\] satisfying
\[{D_{\lambda ,m}}(f - {q_n}) \to 0 \quad as \quad n \to \infty\] immediately.
\end{proof}

Now we define the difference quotient operator ${Q_\lambda }:{D_{\lambda ,m}} \to {D_{\sigma ,m - 1}}$ by
\[({Q_\lambda }f)(z) = \frac{{f(z) - f(\lambda )}}{{z - \lambda }},\quad z \in \mathbb{D}.\]
We equip the space $\widehat{\mathcal{H}}_{\lambda,m}$ with the norm
\begin{equation}
\left\| f \right\|_{\lambda ,m} = D_{\lambda ,m}^{\frac{1}{2}}(f) = D_{\sigma ,m - 1}^{\frac{1}{2}}(Q_\lambda f).
\tag{4.2}
\end{equation}
Let ${\lambda _1}, \cdots ,{\lambda _s}$ be $s$ distinct points on $\mathbb{T}$ and ${c_j} \ge 0$. Set
\[\mu  = {c_1}{\delta _{{\lambda _1}}} + {c_2}{\delta _{{\lambda _2}}} +  \cdots  + {c_s}{\delta _{{\lambda _s}}}.\]
We now consider the case where the measure $\mu$ is a finite sum of Dirac masses.
In this setting, functions in $\widehat{\mathcal{H}}_{\mu,m}$ admit a canonical factorization,
and the modified Taylor polynomials can be constructed via interpolation at finitely many boundary points.

Similar to the treatment of Section 2 in \cite{MR4832456}, each $f \in \widehat{\mathcal{H}}_{\mu,m}$ can be uniquely represented as
\begin{equation}
f(z) = a_0 + a_1 z + \cdots + a_{s-1} z^{s-1} + (z - \lambda_1)(z - \lambda_2) \cdots (z - \lambda_s) g(z), \tag{4.3}
\end{equation}
where $g := Q_{\lambda_1} Q_{\lambda_2} \cdots Q_{\lambda_s} f\in \mathcal{H}_{\sigma ,m - 1}$. We define
\begin{equation}
 \left\| f \right\|_{\mu ,m} := [\sum\limits_{k = 1}^s {{c_k}{D_{\sigma ,m - 1}}({Q_{{\lambda _k}}}f)}]^{\frac{1}{2}}. \tag{4.4}
 \end{equation}
For $f(z) = \sum\limits_{k = m}^\infty  {{a_k}{z^k}} \in \widehat{\mathcal{H}}_{\mu,m} $ and $\Lambda  = \{ {\lambda _j}:1 \le j \le s\} $ be a subset of distinct points in $\mathbb{T}$.
Define
\begin{equation}
P_{m,n}^\Lambda f(z): = \sum\limits_{k = m}^{n + m - s} {{a_k}{z^k}}  + b_{n + m - s + 1}^\Lambda {z^{n + m - s + 1}} +  \cdots  + b_{n + m}^\Lambda {z^{n + m}},~ \text{for each} \ n \ge s. \tag{4.5}
\end{equation}
The polynomial $P_{m,n}^\Lambda f$ coincides with the Taylor partial sum $S_{m,n}f$
up to a modification of the highest $s$ coefficients. By appropriating values of $b_i^\Lambda s$ $(n + m - s + 1 \le i \le n + m)$, we deduce the following equations. That is,
\[P_{m,n}^\Lambda f({\lambda _j}) = f({\lambda _j}), \quad 1 \le j \le s.\tag{4.6}\]

\begin{theorem}\label{tm5.1}
Let $f(z) = \sum\limits_{k = m}^\infty  {{a_k}{z^k}} \in \widehat{\mathcal{H}}_{\mu,m}$, $\Lambda  = \{ {\lambda _j}:1 \le j \le s\} $ be a subset of distinct points in $\mathbb{T}$ and
$$\mu  = {c_1}{\delta _{{\lambda _1}}} + {c_2}{\delta _{{\lambda _2}}} +  \cdots  + {c_s}{\delta _{{\lambda _s}}},$$
where \( c_j > 0 \) $\left( {1 \le j \le s} \right)$. Then the following statements hold.
\begin{itemize}
\item[(i)] The boundary values $f(\lambda_j)$ exist for all $1\le j\le s$.
\item[(ii)] For each $n\ge s$, there exist unique coefficients
$b_{n+m-s+1}^\Lambda,\ldots,b_{n+m}^\Lambda$ in (4.5)
such that (4.6) holds.
\end{itemize}

\begin{proof}
(i): By lemma \ref{lm2.1}, the series $f({\lambda_j})$  converges obviously. By Corollary \ref{co4.5},
 \[
 \begin{cases}
 a = {\lambda _j} b_0, \\
 b_{k - 1} = {\lambda _j} b_k, & \text{for } 1 \le k \le m - 1, \\
 b_{k - 1} - {\lambda _j} b_k = a_k, & \text{for } k \ge m.
 \end{cases}
 \]
 Then
 \begin{align*}
 {S_{m,n}}f({\lambda _j}) &= \sum\limits_{k = m}^{n + m} {{a_k}\lambda _j^k} =\sum\limits_{k = m}^{n + m} {({b_{k - 1}} - {\lambda _j} {b_k})\lambda _j^k} \\
                        &= \sum\limits_{k = m}^{n + m} {{b_{k - 1}}\lambda _j^k}  - \sum\limits_{k = m}^{n + m} {{b_k}\lambda _j^{k + 1}} = {b_{m - 1}}\lambda _j^m - {b_{n + m}}\lambda _j^{n + m + 1}\\
                        &= \sum\limits_{k = m}^{\infty} {{a_k}\lambda _j^k}  - b_{n + m} {\lambda _j^{n + m + 1}},
 \end{align*}
 since Corollary \ref{co5.4} gives us ${b_{m - 1}} = \sum\limits_{k = m}^\infty  {{a_k}\lambda _j^{k - m}} $. This implies $\mathop {\lim }\limits_{n \to \infty } {S_{m,n}}f({\lambda _j}) = \sum\limits_{k = m}^\infty  {{a_k}\lambda _j^k}  = f({\lambda _j})$.\\
 (ii): According to (i), the condition (4.6) is equivalent to
 \[b_{n + m - s + 1}^\Lambda  +  \cdots  + b_{n + m}^\Lambda \lambda _j^{s - 1} = \sum\limits_{{k} = n + m - s + 1}^\infty  {{a_k}\lambda _j^{k - (n + m - s + 1)}} ,\quad 1 \le j \le s.\]
 these equations can be represented as the matrix form $AB = C$, where $A$ is the Vandermonde matrix
 \[
 A = \begin{pmatrix}
 1 & \lambda_1 & \cdots & \lambda_1^{s-1} \\
 1 & \lambda_2 & \cdots & \lambda_2^{s-1} \\
 \vdots & \vdots & \ddots & \vdots \\
 1 & \lambda_s & \cdots & \lambda_s^{s-1}
 \end{pmatrix},
 \]
 and
 \[B = \left( {\begin{array}{*{20}{c}}
 {b_{n + m - s + 1}^\Lambda }\\
 {b_{n + m - s + 2}^\Lambda }\\
 \vdots \\
 {b_{n + m}^\Lambda }
 \end{array}} \right),\quad \quad C = \left( {\begin{array}{*{20}{c}}
 {\sum\limits_{k = n + m - s + 1}^\infty  {{a_k}} \lambda _1^{k - (n + m - s + 1)}}\\
 {\sum\limits_{k = n + m - s + 1}^\infty  {{a_k}} \lambda _2^{k - (n + m - s + 1)}}\\
 \vdots \\
 {\sum\limits_{k = n + m - s + 1}^\infty  {{a_k}} \lambda _s^{k - (n + m - s + 1)}}
 \end{array}} \right).\]
 Noting that $\det (A) = \prod\limits_{1 \le i < j \le s} {({\lambda _j} - {\lambda _i})} $, and that
 $\{ {\lambda _j}:1 \le j \le s\}$ is a subset of distinct points in $\mathbb{T}$, we conclude that $\det (A) \ne 0$ and $A$ is invertible.
 This proves that the coefficients \( b^{\Lambda}_{n + m - s + 1}, \ldots, b^{\Lambda}_{n + m} \) are uniquely determined.
\end{proof}
\end{theorem}

\begin{theorem}\label{tm5.6}
Let $f \in \widehat{\mathcal{H}}_{\mu,m}$, $\Lambda  = \{ {\lambda _j}:1 \le j \le s\} $ be a subset of distinct points in $\mathbb{T}$ and
$$\mu  = {c_1}{\delta _{{\lambda _1}}} + {c_2}{\delta _{{\lambda _2}}} +  \cdots  + {c_s}{\delta _{{\lambda _s}}},$$
where $c_j > 0 \left( {1 \le j \le s} \right)$. Then
\[{\left\| {P_{m,n}^\Lambda f - f} \right\|_{\mu,m }} \to 0 \quad as \quad n \to \infty. \]
\end{theorem}

\begin{proof}
When $\Lambda = {\Lambda _1}= \{ {\lambda _1}\} $ is a singleton, the result has been proven directly in Corollary \ref{co5.4}. We obtain
\[P_{m,n}^{{\Lambda _1}}f(z) = \sum\limits_{k = m}^{n + m - 1} {{a_k}{z^k}}  + \left( {\sum\limits_{k = n + m}^\infty  {{a_k}\lambda _1^{k - n - m}} } \right){z^{n + m}},\]
which also implies that
\[{\left\| {P_{m,n}^{{\Lambda _1}}f(z) - f(z)} \right\|_{{\lambda _1},m}} \to 0, \quad n \to \infty.\]
Suppose the result holds for \( s - 1 \) distinct points in \( \mathbb{T} \), and let \( \Lambda = \{\lambda_1, \ldots, \lambda_s\} \) be a set of \( s \) distinct points in \( \mathbb{T} \). To proceed, we consider the following two subsets of \( \Lambda \),
\[
\Lambda_1 = \{\lambda_1, \lambda_3, \ldots, \lambda_s\}, \quad
\Lambda_2 = \{\lambda_2, \lambda_3, \ldots, \lambda_s\}, \text{consisting of} \ s - 1 \text{ distinct points}.
\]
Let $P_{m,n - 1}^{{\Lambda _1}}f$ and $P_{m,n - 1}^{{\Lambda _2}}f$ denote the corresponding modified Taylor polynomials. By our assumption
\begin{align*}
P_{m,n - 1}^{{\Lambda _1}}f &= \sum\limits_{k = m}^{n + m - s} {{a_k}{z^k}}  + b_{n + m - s + 1}^{{\Lambda _1}}{z^{n + m - s + 1}} +  \cdots  + b_{n + m - 1}^{{\Lambda _1}}{z^{n + m - 1}},\\
P_{m,n - 1}^{{\Lambda _2}}f &= \sum\limits_{k = m}^{n + m - s} {{a_k}{z^k}}  + b_{n + m - s + 1}^{{\Lambda _2}}{z^{n + m - s + 1}} +  \cdots  + b_{n + m - 1}^{{\Lambda _2}}{z^{n + m - 1}}.
\end{align*}
Furthermore,
\begin{align}
P^{\Lambda_1}_{n-1} f(\lambda_j) &= f(\lambda_j), \hspace{2em} j \in \{1,3,4,\ldots,s\}, \tag{4.7} \\
P^{\Lambda_2}_{n-1} f(\lambda_j) &= f(\lambda_j), \hspace{2em} j \in \{2,3,4,\ldots,s\} \tag{4.8}
\end{align}
such that
\begin{equation}
{\left\| {P_{m,n - 1}^{{\Lambda _1}}f - f} \right\|_{{\mu _1},m}} \to 0\quad and \quad {\left\| {P_{m,n - 1}^{{\Lambda _2}}f - f} \right\|_{{\mu _2},m}} \to 0 \tag{4.9}
\end{equation}
as $n\to\infty$, where
\[
\mu_1 = c_1 \delta_{\lambda_1} + c_3 \delta_{\lambda_3} + \cdots + c_s \delta_{\lambda_s},
\]
and
\[
{\mu _2} = {c_2}{\delta _{{\lambda _2}}} + {c_3}{\delta _{{\lambda _3}}} +  \cdots  + {c_s}{\delta _{{\lambda _s}}}.
\]
Set \( a = (\lambda_2 - \lambda_1)^{-1} \) and \( b = (\lambda_1 - \lambda_2)^{-1} \), then
\begin{equation}
a(z - \lambda_1) + b(z - \lambda_2) = 1, \quad z \in \mathbb{C}. \tag{4.10}
\end{equation}
Define
\[{h_{m,n}}(z) = a(z - {\lambda _1})P_{m,n - 1}^{{\Lambda _2}}f(z) + b(z - {\lambda _2})P_{m,n - 1}^{{\Lambda _1}}f(z).\]
According to (4.9), it is obvious that the first $n+s-1$ coefficients of ${h_{m,n}}$ coincide with the first $n+s-1$ coefficients of $f$. Moreover, applying \((4.7)\) and \((4.8)\), we obtain
\[{h_{m,n}}({\lambda _1}) = P_{m,n - 1}^{{\Lambda _1}}f({\lambda _1}) = f({\lambda _1}),\]
\[{h_{m,n}}({\lambda _2}) = P_{m,n - 1}^{{\Lambda _2}}f({\lambda _2}) = f({\lambda _2}).\]
For each ${\lambda _j} \in \left\{ {{\lambda _3},{\lambda _{4,}} \cdots {\lambda _s}} \right\}$,
\[{h_{m,n}}({\lambda _j}) = a({\lambda _j} - {\lambda _1})P_{m,n - 1}^{{\Lambda _2}}f({\lambda _j}) + b({\lambda _j} - {\lambda _2})P_{m,n - 1}^{{\Lambda _1}}f({\lambda _j})=f({\lambda_j}).\]
By Theorem \ref{tm5.1} (ii), if such $P_{m,n}^\Lambda f$ exists, it must be unique. Thus, $P_{m,n}^\Lambda f={h_{m,n}}$.
Now we aim to show that $${\left\| {{h_{m,n}} - f} \right\|_{\mu ,m}} \to 0  \quad as \quad n \to \infty.$$
According to (4.4) and (4.9), we know that\\
\begin{equation}
{D_{\sigma ,m - 1}}({Q_{{\lambda _k}}}(P_{m,n - 1}^{{\Lambda _1}}f - f)) \to 0  \quad \text{for} \ 1 \leq k \leq s \ \text{~with~} \ k \ne 2 , \tag{4.11}
\end{equation}
and
\begin{equation}
{D_{\sigma ,m - 1}}({Q_{{\lambda _k}}}(P_{m,n - 1}^{{\Lambda _2}}f - f)) \to 0 \quad \text{for} \ 2 \leq k \leq s. \tag{4.12}
\end{equation}
By \((4.10)\), we have
\[{h_{m,n}}(z) - f(z) = a(z - {\lambda _1})(P_{m,n - 1}^{{\Lambda _2}}f(z) - f(z)) + b(z - {\lambda _2})(P_{m,n - 1}^{{\Lambda _1}}f(z) - f(z)).\]
Noting that
$$P_{m,n - 1}^{{\Lambda _1}}f({\lambda _k}) - f({\lambda _k}) = P_{m,n - 1}^{{\Lambda _2}}f({\lambda _k}) - f({\lambda _k}) = 0 \quad \text{for} \ 3 \le k \le s,$$
we have
\[{Q_{{\lambda _k}}}({h_{m,n}} - f) = a(z - {\lambda _1}){Q_{{\lambda _k}}}(P_{m,n - 1}^{{\Lambda _2}}f - f) + b(z - {\lambda _2}){Q_{{\lambda _k}}}(P_{m,n - 1}^{{\Lambda _1}}f - f).\]
Combining this with \((4.11)\) and \((4.12)\)
\begin{align*}
{D_{\sigma ,m - 1}}({Q_{{\lambda _k}}}({h_{m,n}} - f)) &\le 4\left| a \right|{D_{\sigma ,m - 1}}({Q_{{\lambda _k}}}(P_{m,n - 1}^{{\Lambda _2}}f - f)) + 4\left| b \right|{D_{\sigma ,m - 1}}({Q_{{\lambda _k}}}(P_{m,n - 1}^{{\Lambda _1}}f - f))\\
&\to 0
\end{align*}
as $ n \to \infty $.

We still need to show that it is also valid for \( k = 1 \) and \( k = 2 \). Since ${h_{m,n}}({\lambda _1}) - f({\lambda _1}) = 0 $, we obtain
\begin{align*}
{Q_{{\lambda _1}}}({h_{m,n}} - f)(z) &= \frac{{a(z - {\lambda _1})(P_{m,n - 1}^{{\Lambda _2}}f(z) - f(z)) + b(z - {\lambda _2})(P_{m,n - 1}^{{\Lambda _1}}f(z) - f(z))}}{{z - {\lambda _1}}}\\
                                   &=a(P_{m,n - 1}^{{\Lambda _2}}f - f)(z) + b(z - {\lambda _2}){Q_{{\lambda _1}}}(P_{m,n - 1}^{{\Lambda _1}}f - f)(z),
\end{align*}
which, together with \((4.9)\) and \((4.11)\), indicates that
\begin{align*}
{D_{\sigma ,m - 1}}({Q_{{\lambda _1}}}({h_{m,n}} - f)) &\le 2\left| a \right|{D_{\sigma ,m - 1}}(P_{m,n - 1}^{{\Lambda _2}}f - f) + 4\left| b \right|{D_{\sigma ,m - 1}}({Q_{{\lambda _1}}}(P_{m,n - 1}^{{\Lambda _1}}f - f))\\
&\to 0
\end{align*}
as $ n \to \infty $.

Similarly discussion also implies that
\[{D_{\sigma ,m - 1}}({Q_{{\lambda _2}}}({h_{m,n}} - f)) \to 0 \quad as \quad n \to \infty .\]
Consequently, we conclude that
\[\left\| {{h_{m,n}} - f} \right\|_{\mu ,m}^2 = \sum\limits_{k = 1}^s {{c_k}{D_{\sigma ,m - 1}}({Q_{{\lambda _k}}}({h_{m,n}} - f))}  \to 0 \quad as \quad n \to \infty .\]

This completes the induction and establishes a Fej\'er-type summability
scheme for higher-order weighted Dirichlet spaces associated with
finite Dirac measures.

\end{proof}

In summary, when the weight measure $\mu$ is a finite combination of Dirac masses,
the abstract approximation scheme developed earlier reduces to an explicit interpolation procedure.
The resulting modified Taylor polynomials admit closed-form expressions and converge
in the norm of $\widehat{\mathcal{H}}_{\mu,m}$, thereby providing a concrete and constructive realization
of Fej\'er-type summability in higher-order weighted Dirichlet spaces.
\\

\textbf{Data availability} This manuscript has no associated data.
	\\
	
	\textbf{Competing Interests} The authors have no relevant financial or non-financial interests to disclose.

	\begin{center}

\end{center}
\end{document}